\documentclass[12pt]{amsart}

\usepackage{amsmath,amssymb,amsthm,bbm,bm,color,tikz}
\usepackage[ruled,vlined]{algorithm2e}
\usepackage[mathscr]{euscript}

\usepackage[margin=1in]{geometry}

\theoremstyle{definition}
\newtheorem{thm}{Theorem}[section]

\newtheorem{lem}[thm]{Lemma}

\newtheorem{rmk}[thm]{Remark}

\DeclareMathOperator{\NVol}{NVol}
\DeclareMathOperator{\vol}{vol}
\DeclareMathOperator{\conv}{conv}

\DeclareMathOperator{\PQ}{PQ}

\newcommand{\alphatri}{\alpha^{\triangle}}
\newcommand{\betatri}{\beta^{\triangle}}
\newcommand{\gammatri}{\gamma^{\triangle}}
\newcommand{\gammatrip}{\gamma^{'\triangle}}

\newcommand{\scrAtri}{\mathscr{A}^{\triangle}}

\newcommand{\scrBtri}{\mathscr{B}^{\triangle}}

\newcommand{\scrCtri}{\mathscr{C}^{\triangle}}

\newcommand{\calM}{\mathcal{M}}
\newcommand{\calN}{\mathcal{N}}
\newcommand{\calP}{\mathcal{P}}
\newcommand{\fkD}{\mathfrak{D}}

\newcommand{\e}{\bm{e}}

\newcommand{\RR}{\mathbb{R}}
\newcommand{\ZZ}{\mathbb{Z}}

\newcommand{\scrXcir}{\mathscr{X}}
\newcommand{\scrYcir}{\mathscr{Y}}
\newcommand{\scrZcir}{\mathscr{Z}}

\newcommand{\APQ}{\nabla^{\PQ}}

\newcommand\commentout[1]{}

\makeatletter
\def\@tocline#1#2#3#4#5#6#7{\relax
  \ifnum #1>\c@tocdepth 
  \else
\par \addpenalty\@secpenalty\addvspace{#2}%
\begingroup \hyphenpenalty\@M
\@ifempty{#4}{%
  \@tempdima\csname r@tocindent\number#1\endcsname\relax
}{%
  \@tempdima#4\relax
}%
\parindent\z@ \leftskip#3\relax \advance\leftskip\@tempdima\relax
\rightskip\@pnumwidth plus4em \parfillskip-\@pnumwidth
#5\leavevmode\hskip-\@tempdima
  \ifcase #1
   \or\or \hskip 1em \or \hskip 2em \else \hskip 3em \fi%
  #6\nobreak\relax
\hfill\hbox to\@pnumwidth{\@tocpagenum{#7}}\par
\nobreak
\endgroup
  \fi}
\makeatother

\begin{document}


\title{Normalized Volumes of Type-PQ Adjacency Polytopes for Certain Classes of Graphs}

\author{Robert Davis}
\address{Department of Mathematics\\
 Colgate University\\
 Hamilton, NY USA}
\email{rdavis@colgate.edu}

\author{Joakim Jakovleski}
\address{Department of Mathematics\\
 Colgate University\\
 Hamilton, NY USA}
\email{jjakovleski@colgate.edu}

\author{Qizhe Pan}
\address{Department of Mathematics\\
 Colgate University\\
 Hamilton, NY USA}
\email{qpan@colgate.edu}

\thanks{The authors were supported in part by NSF grant DMS-1922998 and by Colgate University.}

\maketitle



\begin{abstract}
	The type-PQ adjacency polytope associated to a simple graph is a $0/1$-polytope containing valuable information about an underlying power network.
	Chen and the first author have recently demonstrated that, when the underlying graph $G$ is connected, the normalized volumes of the adjacency polytopes can be computed by counting sequences of nonnegative integers satisfying restrictions determined by $G$.
	This article builds upon their work, namely by showing that one of their main results -- the so-called ``triangle recurrence'' -- applies in a more general setting.
	Formulas for the normalized volumes when $G$ is obtained by deleting a path or a cycle from a complete graph are also established.
\end{abstract}

\section{Introduction}

Let $N$ be a positive integer.
A \emph{polytope} $P \subseteq \RR^N$ is the convex hull of finitely many points $v_1,\dots,v_d \in \RR^N$, that is,
\[
	P = \conv\{v_1,\dots,v_d\} = \left\{ \sum_{i=1}^d \lambda_iv_i \mid \lambda_1,\dots,\lambda_d \geq 0, \sum_{i=1}^d \lambda_i = 1\right\}.
\]
To a simple, undirected graph $G$ on $[N] = \{1,\dots,N\}$, the associated \emph{type-PQ adjacency polytope} is
\[
	\APQ_G = \conv\{(\e_i, \e_j) \in \RR^{2N} \mid i=j \text{ or } ij \in E(G)\}
\]
where $\e_i$ is the $i^{th}$ standard basis vector of $\RR^N$ and $E(G)$ is the set of edges of $G$.
Note that since $G$ is undirected, an edge $ij$ can also be written $ji$, meaning that a single edge of $G$ produces two points $(\e_i,\e_j)$ and $(\e_j,\e_i)$ used in the construction of $\APQ_G$.

Type-PQ adjacency polytopes arise in the study of power-flow equations of electrical networks \cite{ChenMehtaPQ}.
These polytopes are in contrast to type-PV adjacency polytopes, which arise in the study of systems of interconnected oscillators.
The first author, together with Chen and others, have recently studied their combinatorial structures and triangulations with a view towards their use in homotopy continuation methods \cite{DavisChenKorchevskaia, ChenDavisMehta, DavisChen}.
Their work has sparked a flurry of results \cite{BraunBruegge, ManyFaces, DAliEtAl, OhsugiTsuchiyaPV, OhsugiTsuchiyaPQ} using various algebraic and combinatorial techniques, often under the name of \emph{symmetric edge polytopes}, which were studied previously with a view towards Ehrhart-theoretic results \cite{HigashitaniJochemkoMichalek, MatsuiEtAl}.

To state the main results of this article, we must define several additional terms.
The \emph{dimension} of a polytope, denoted $\dim(P)$, is the dimension of the affine linear subspace spanned by $P$.
One of the crucial pieces of information regarding $\APQ_G$ is its \emph{normalized volume}, $\NVol(\APQ_G) = \dim(\APQ_G)!\vol(\APQ_G)$, where $\vol(P)$ is the (relative) Euclidean volume of $P$.
Part of what makes this value of interest is the fact that, if $P$ is the convex hull of points in $\ZZ^N$, then $\NVol(P)$ is always a positive integer.
Computing the normalized volume can then be approached by attempting to find a set $S$ for which $\NVol(P) = |S|$, ideally so that the structure of $S$ allows for a more feasible computation of $|S|$.
This article uses such an approach, first implemented in this setting in \cite{DavisChen}, to produce formulas for $\NVol(\APQ_G)$ for several classes of graphs.
Our main results are the following.

\setcounter{section}{2}
\setcounter{thm}{8}

\begin{thm}\label{thm: generalized triangle rec}
	Let $N \geq 2$.
	If $\calM$ is an $M$-element matching of $K_N$, then
	\[
		\NVol(\APQ_{K_N \triangle \calM}) = 3^M\binom{2(N-1)}{N-1}.
	\]
\end{thm}

\setcounter{section}{3}
\setcounter{thm}{1}
\begin{thm}\label{thm: delete path from K_N}
	Let $N \geq 4$ and $0 \leq M < N$.
	If $P$ is a length-$M$ path in $K_N$, then
	\[
		\NVol(\APQ_{K_N \setminus E(P)}) = \binom{2(N-1)}{N-1} - (2N - 4)(M - 1) + 4.
	\]
\end{thm}

\setcounter{thm}{4}
\begin{thm}\label{thm: delete cycle from K_N}
	Let $N \geq 5$ and $0 \leq M \leq N$.
	Denote by $E_M^{\circ}$ the set of edges of an $M$-cycle in $K_N$.
	For all such choices of $N$ and $M$ and any $E_M^{\circ}$,
	\[
		\NVol(\APQ_{K_N \setminus E_M^\circ}) =
			\begin{cases} 
				\binom{2(N-1)}{N-1} - 2M(N-2) &  \text{ if }M \neq 4 \\
				\binom{2(N-1)}{N-1} - 2(N+1)(N-2) & \text{ if } M = 4.
			\end{cases}
	\]
\end{thm}
In Section~\ref{sec: background}, we give a brief overview of how computing normalized volumes of $\NVol(\APQ_G)$ can be interpreted as a combinatorial problem.
Within this section, we adapt an argument given in \cite{DavisChen} to prove Theorem~\ref{thm: generalized triangle rec}. 
Section~\ref{sec: deleting from complete} examines how $\NVol(\APQ_G)$ relates to $\APQ_{K_N}$ when $G$ is obtained from $K_N$ by deleting the edges of a path or a cycle. 
It is in that section that we prove Theorems~\ref{thm: delete path from K_N} and \ref{thm: delete cycle from K_N}.

\setcounter{section}{1}
\setcounter{thm}{1}
\section{Background}\label{sec: background}

One key insight provided in \cite{DavisChen} is that $\NVol(\APQ_G)$ can be computed by determining the number of sequences of positive integers satisfying conditions arising from $G$.
To describe the sequences and the constraints placed upon them, we need to first define a number of notions.

For a positive integer $N$, let $[\overline{N}] = \{\overline 1,\dots, \overline N\}$, and define $K_{N,\overline N}$ to be the complete bipartite graph with partite sets $[N]$ and $[\overline N]$.
If $G$ is a graph and $v$ is a vertex of $G$, then we denote by $\calN_G(v)$ the \emph{neighbors} of $v$ in $G$. 
When $G \subseteq K_{N, \overline N}$, we say a sequence $(c_1,\dots,c_N) \in \ZZ_{\geq 0}^N$ is \emph{$G$-draconian sequence} if $\sum c_i = N-1$ and, for any choice of indices $1 \leq i_1 < \cdots < i_k \leq N$, the inequality
\begin{equation}\label{eq: draconian inequality}
	c_{i_1} + \cdots + c_{i_k} < \left|\bigcup_{j=1}^k \calN_G(i_j)\right|
\end{equation}
is satisfied.
We will often call the inequality \eqref{eq: draconian inequality} the \emph{$G$-draconian inequality corresponding to $i_1,\dots,i_k$}, or corresponding to $c_{i_1},\dots,c_{i_k}$.
Often, the graph $G$ and the indices $i_1,\dots,i_k$ are understood from context, and we simply call \eqref{eq: draconian inequality} a \emph{draconian inequality}.
For similar reasons, if $G$ is either understood from or irrelevant to the discussion, we may call $c$ a \emph{draconian sequence}. 

Draconian sequences were first studied by Postnikov in relation to a large class of polytopes called \emph{generalized permutohedra} and a generalization of Hall's Matching Theorem \cite{Postnikov2009}.
He establishes formulas for normalized volumes of generalized permutohedra in terms of draconian sequences, though these formulas are generally nonalgebraic.
Unfortunately, determining the number of $G$-draconian sequences for an arbitrary $G \subseteq K_{N,\overline{N}}$ is computationally expensive, and typically requires determining the set of draconian sequences themselves.
In the context of finding solutions to power-flow equations, therefore, finding algebraic formulas for the number of draconian sequences, when possible, is a significant improvement.

Since draconian sequences are only associated to bipartite graphs, we must relate an arbitrary simple graph to a bipartite graph in a controlled way.
For a simple graph $G$ on $[N]$, define $D(G)$ to be the subgraph of $K_{N,\overline N}$ whose edges are $\{i, \overline i\}$ for each $i \in [N]$ and $\{i, \overline j\}$ and $\{j, \overline i\}$ for each edge $ij$ in $G$.
Denote by $\fkD(G)$ the set of $D(G)$-draconian sequences.
The following result is a crucial connection between type-PQ adjacency polytopes and draconian sequences.

\begin{thm}[{\cite[Theorem 2.8]{DavisChen}}]\label{thm: translation}
	For any connected graph $G$ on $[N]$, $\NVol(\APQ_G) = |\fkD(G)|$.
\end{thm}

We note that if $G$ is disconnected, then the above result may not hold, but an alternate approach \cite[Proposition 3.2]{DavisChen} shows that $\NVol(\APQ_G)$ is a product of factors of the form $|\fkD(G_i)|$ where each $G_i$ is a connected component of $G$. 
Thus, we may always assume without loss of generality that $G$ is connected.
A convenient result that we will also frequently use is the following.

\begin{rmk}[{\cite[Remark 2.9]{DavisChen}}]\label{rmk: relabeling invariant}
	The normalized volume of $\APQ_G$ is invariant under permutation of vertices.
	Hence, $|\fkD(G)|$ is also invariant under relabeling of the vertices of $G$.
\end{rmk}


\subsection{Extending the triangle recurrence}\label{sec: extending triangles}

Let $G$ be a connected graph and let $e = uv$ be an edge of $G$.
We denote the vertex set of $G$ by $V(G)$.
In \cite{DavisChen}, the authors defined the construction $G \triangle e$ to be the graph on the vertex set $V(G) \cup \{w_e\}$, where $w_e$ is a new vertex, and the edge set $E(G) \cup \{uw_e, vw_e\}$.
Under certain conditions, we can describe $\NVol(\APQ_{G \triangle e})$ in terms of $\NVol(\APQ_G)$.

\begin{thm}[{\cite[Theorem 3.18]{DavisChen}}]\label{thm: triangle recurrence}
	Let $G$ be a connected graph on $[N]$ for which $e = uv$ is an edge with $\deg_G(u) = 2$.
	If $\deg_G(v) = 2$ or if the neighbors of $u$ are neighbors of each other, then
	\begin{equation}\label{eq: triple}
		\NVol(\APQ_{G \triangle e}) = 3\NVol(\APQ_G).
	\end{equation}
\end{thm}

We refer to the above recurrence as \emph{the triangle recurrence}.
Experimental data suggests that the conditions on the triangle recurrence can be relaxed considerably, although a complete characterization of when \eqref{eq: triple} holds remains elusive.
The triangle recurrence relies on several lemmas, some of which we will need, and therefore state below.

\begin{lem}[{\cite[Lemma 3.14]{DavisChen}}]\label{lem: first rec lem tri}
Let $G$ be any connected graph on $[N]$ and $e$ any edge.
	If $c \in \fkD(G)$, then $\alphatri(c) \in \fkD(G\triangle e)$ where $\alphatri(c) = (c,1)$.
	Moreover, $\alphatri$ is injective. \qed
\end{lem}

\begin{lem}[{\cite[Lemma 3.15]{DavisChen}}]\label{lem: second rec lem tri}
Let $G$ be a connected graph on $[N]$ and let $e = uv$ be any edge.
	If $c \in \fkD(G)$, then $\betatri(c) \in \fkD(G\triangle e)$ where
	\[
		\betatri(c) = \alphatri(c) + \e_u - \e_{N+1}.
	\]
	Additionally, $\betatri$ is injective. \qed
\end{lem}

We have occasional need for the notation $\scrBtri_{G}(e)$ to denote the image of $\gammatri(e)$.
The benefit of these lemmas is their ability to apply to any connected graph. 
A third lemma is needed to establish the triangle recurrence, but to prove it requires additional restrictions on $G$.
The lemma we need will be an adaptation of the following.

\begin{lem}[{\cite[Lemma 3.16]{DavisChen}}]\label{lem: third rec lem tri}
Let $G$ be a connected graph on $[N]$ and let $e = uv$ be any edge for which $\deg_G(u) = 2$.
	If $c \in \fkD(G)$, then $\gammatri(c) \in \fkD(G\triangle e)$ where
	\[
		\gammatri(c) =
			\begin{cases}
				\alphatri(c) + \e_v - \e_{N+1} & \text{ if not in } \scrBtri_G(e) \\
				\alphatri(c) - \e_u + \e_{N+1} & \text{ otherwise}.
			\end{cases}
	\]
	Additionally, $\gammatri$ is injective. \qed
\end{lem}

Much of the work we will need to do involves showing that the conclusion to Lemma~\ref{lem: third rec lem tri} will still hold in our more general setting. 
The rest of the work will be in showing that the images of $\alphatri(G), \betatri(G)$, and $\gammatri(G)$ are disjoint, and that each appropriate $D(G)$-draconian sequence is in the image of one of these functions.

For our adaptation of Lemma~\ref{lem: third rec lem tri}, we slightly extend the definition of $G \triangle e$ as follows:
If $F$ is a subset of edges of $G$, then set
\[
	G \triangle F = \bigcup_{e \in F} G \triangle e.
\]
In other words, $G \triangle F$ is the graph obtained by taking each edge $e = uv$ in $F$, creating a new vertex $w_e$, and adding the edges $uw_e$ and $vw_e$.


\begin{lem}\label{lem: third rec lem tri new}
	Let $\calM$ be an $M$-element matching of $K_N$ and $e = uv$ an edge of $K_N$ such that $\calM ' = \calM \cup \{e\}$ is also a matching of $K_N$.
	If $c \in \fkD(K_N \triangle \calM)$, then $\gammatri(c) \in \fkD(K_N \triangle \calM')$ where
	\[
		\gammatri(c) = 
			\begin{cases} 
				\alphatri(c) + \e_v - \e_{w_e} & \text{if not in } \scrBtri_{K_N \triangle \calM}(e) \\ 
				\alphatri(c) - \e_u + \e_{w_e} & \text{otherwise}.
			\end{cases}
	\]
Additionally, $\gammatri$ is injective.
\end{lem}

\begin{proof}
	Let $\calM$ be an $M$-element, non-maximal matching of $K_N$.
	By Remark~\ref{rmk: relabeling invariant}, we may assume, without loss of generality, that $V(K_N) = [N]$, $V(K_N \triangle \calM) = [N+M]$, $e = \{N+M-1,N+M\}$, and the new vertex added to $K_N \triangle \calM$ is $N+M+1$.
	
	That $\gammatri$ is injective is clear from the definition.
	We wish to show, then, that if $c = (c_1,\dots, c_{N+M}) \in \fkD(G \triangle \calM)$, then $\gammatri(c) \in \fkD(G \triangle \calM')$, where
	\[
		\gammatri(c) = 
			\begin{cases} 
				(c_1,\dots, c_{N+M-1}, c_{N+M}+1, 0) & \text{if not in } \scrBtri_{K_N \triangle \calM}(e) \\ 
				(c_1,\dots, c_{N+M-1}-1, c_{N+M}, 2) & \text{otherwise}.
			\end{cases}
	\]
	Note that in each case, the sum of all entries is $N+M$, as needed.
	
	First suppose that $\gammatri(c) = (c_1,\dots, c_{N+M-1}, c_{N+M}+1, 0)$.
	For notational convenience, let the entries of $\gammatri(c)$ be denoted by $(\gammatri_1,\dots, \gammatri_{N+M+1})$.
	We will verify that the required $D(K_N \triangle \calM')$-draconian inequalities hold.

	Let $1 \leq i_1 < \cdots < i_k \leq N+M+1$.
	If $i_k < N+M-1$, then $\gammatri_{i_j} = c_{i_j}$ for each $j$, and, since $c \in \fkD(K_N \triangle \calM)$, we have
	\[
		\gammatri_{i_1} +\dots+ \gammatri_{i_k} = c_{i_1} +\dots+ c_{i_k} < \left|\bigcup_{j=1}^k \calN_{D(K_N \triangle \calM)}(i_j)\right| = \left|\bigcup_{j=1}^k \calN_{D(K_N \triangle \calM')}(i_j)\right|.
	\]
	If $i_k \in \{N+M-1,N+M\}$, then 
	\[
		\begin{aligned}
			\gammatri_{i_1} +\dots+ \gammatri_{i_k} 
				&\leq c_{i_1} +\dots+ c_{i_k} + 1 \\
				&< \left|\bigcup_{j=1}^k \calN_{D(K_N \triangle \calM)}(i_j)\right| + |\{\overline{N+M+1}\}| \\
				&= \left|\bigcup_{j=1}^k \calN_{D(K_N \triangle \calM')}(i_j)\right|.
		\end{aligned}
	\]
	Lastly, if $i_k = N+M+1$, then, by the previous cases,
	\[
		\begin{aligned}
			\gammatri_{i_1} +\dots+ \gammatri_{i_k} 
				&= c_{i_1} +\dots+ c_{i_{k-1}} \\
				&< \left|\bigcup_{j=1}^{k-1} \calN_{D(K_N \triangle \calM')}(i_j)\right| \\
				&\leq \left|\bigcup_{j=1}^k \calN_{D(K_N \triangle \calM')}(i_j)\right|.
		\end{aligned}
	\]
	Thus, $\gammatri(c) \in \fkD(K_N \triangle \calM')$ if it is not already a member of $\scrBtri_{K_N \triangle \calM}(e)$.
	
	Now consider the case in which $\gammatri(c) = (c_1,\dots, c_{N+M-1}-1, c_{N+M}, 2)$. 
	Here, we know $(c_1,\dots, c_{N+M-1}, c_{N+M}+1, 0) \in \scrBtri_{K_N \triangle \calM}(e)$, so that $c_{N+M-1} \geq 1$.
	This time, denote the entries of $\gammatri(c)$ by $(\gammatrip_1,\dots, \gammatrip_{N+M+1})$ and let $1 \leq i_1 < \dots < i_k \leq N+M+1$.
	If $i_k < N+M+1$, then
	\[
		\gammatrip_{i_1} +\dots+ \gammatrip_{i_k} \leq c_{i_1} +\dots+ c_{i_k} < \left|\bigcup_{j=1}^k \calN_{D(K_N \triangle \calM)}(i_j)\right| \leq \left|\bigcup_{j=1}^k \calN_{D(K_N \triangle \calM')}(i_j)\right|.
	\]
	
	If $i_k = N+M+1$, then there are two cases to consider.
	If $i_j > N$ for each $j < k$, then no two $i_j$ and $i_{j'}$ share a neighbor in $D(K_N \triangle \calM')$.
	Therefore,
	\[
		\gammatrip_{i_1} +\dots+ \gammatrip_{i_k} \leq 2 + \cdots + 2 < \left|\bigcup_{j=1}^k \calN_{D(K_N \triangle \calM')}(i_j)\right|.
	\]
	Otherwise, suppose $i_j \leq N$ for some $j < k$.
	If $i_{\ell} = N+M-1$, for some $\ell$, then
	\[
		\begin{aligned}
			\gammatrip_{i_1} +\dots+ \gammatrip_{i_k} 
				&= c_{i_1} +\dots+ c_{i_{k-1}} - 1 + 2 \\
				&< \left|\bigcup_{j=1}^{k-1} \calN_{D(K_N \triangle \calM)}(i_j)\right| + |\{\overline{N+M+1}\}| \\
				&\leq \left|\bigcup_{j=1}^k \calN_{D(K_N \triangle \calM')}(i_j)\right|.
		\end{aligned}
	\]
	If $i_j \neq N+M-1$ for all $j$, then we use the fact that 
	\[
		\calN_{K_N \triangle \calM}(N+M-1) \subseteq \bigcup_{j=1}^{k-1} \calN_{D(K_N \triangle \calM)}(i_j)
	\]
	to obtain
	\[
		\begin{aligned}
			\gammatrip_{i_1} +\dots+ \gammatrip_{i_k} &\leq c_{i_1} +\dots+ c_{i_{k-1}} + c_{N+M-1} - 1 + 2 \\
				&< \left|\left(\bigcup_{j=1}^{k-1} \calN_{D(K_N \triangle \calM)}(i_j)\right) \cup \calN_{K_N \triangle \calM}(N+M-1) \right| + |\{\overline{N+M+1}\}| \\
				&\leq \left|\bigcup_{j=1}^k \calN_{D(K_N \triangle \calM')}(i_j)\right|.
		\end{aligned}
	\]
	Since all draconian inequalities hold, we have $\gammatri(c) \in \fkD(K_N \triangle \calM')$.
\end{proof}

The following lemma is straightforward from the definitions of $\alphatri$, $\betatri$, and $\gammatri$, so we omit its proof.
In it, we use $\scrAtri_{K_N \triangle \calM}(e)$ to denote the image of $\alphatri$ and $\scrCtri_{K_N \triangle \calM}(e)$ to denote the image of $\gammatri$.

\begin{lem}\label{lem: pairwise disjoint}
	Let $\calM$ be matching of $K_N$ and $e = uv$ an edge of $K_N$ such that $\calM ' = \calM \cup \{e\}$ is also a matching of $K_N$.
	The sets $\scrAtri_{K_N \triangle \calM}(e), \scrBtri_{K_N \triangle \calM}(e)$, and $\scrCtri_{K_N \triangle \calM}(e)$ are pairwise disjoint. \qed
\end{lem}

\begin{thm}\label{thm: triangles on matching}
	Let $N \geq 2$.
	If $\calM$ is an $M$-element matching of $K_N$, then
	\[
		\NVol(\APQ_{K_N \triangle \calM}) = 3^M\binom{2(N-1)}{N-1}.
	\]
\end{thm}

\begin{proof}
	We will proceed by induction.
	The case of $\calM = \emptyset$ is exactly \cite[Proposition 2.10]{DavisChen}.
	We assume that the conclusion holds for any non-maximal, $M$-element matching of $K_N$.
	
	Let $\calM$ be an arbitrary non-maximal, $M$-element matching of $K_N$.
	Since $\calM$ is non-maximal, there is an edge $e$ of $K_N$ such that $\calM' = \calM \uplus \{e\}$ is a matching of $K_N$ as well.
	By Remark~\ref{rmk: relabeling invariant}, we may freely relabel the vertices of $K_N \triangle \calM$ such that $e = \{N+M-1, N+M\}$ and such that the unique vertex of $K_N \triangle \calM'$ not appearing in $K_N \triangle \calM$ is $N+M+1$. 
	Applying Lemmas~\ref{lem: first rec lem tri}, \ref{lem: second rec lem tri}, \ref{lem: third rec lem tri new}, and \ref{lem: pairwise disjoint}, we have 
	\[
		\scrAtri_{K_N \triangle \calM}(e) \uplus \scrBtri_{K_N \triangle \calM}(e) \uplus \scrCtri_{K_N \triangle \calM}(e) \subseteq \fkD(K_N \triangle \calM').
	\]
	Hence, we only need to establish the reverse inclusion. 
	
	Let $d = (d_1,\dots, d_{N+M+1}) \in \fkD(K_N \triangle \calM')$.
	We will show that $d$ is in one of $\scrAtri_{K_N \triangle \calM}(e)$, $\scrBtri_{K_N \triangle \calM}(e)$, $\scrCtri_{K_N \triangle \calM}(e)$ by proving each of the following claims:
	\begin{enumerate}
		\item if $d_{N+M+1} = 1$, then $(d_1,\dots, d_{N+M}) \in \fkD(K_N \triangle \calM)$;
		\item if $d_{N+M+1} = 0$, then one of $(d_1,\dots, d_{N+M-1}-1, d_{N+M})$ and $(d_1,\dots, d_{N+M-1},  d_{N+M}-1)$ is in $\fkD(K_N \triangle \calM)$;
		\item if $d_{N+M+1} = 2$, then $(d_1,\dots, d_{N+M-1}+1, d_{N+M}+1, 0) \in \scrBtri_{K_N \triangle \calM}(e)$ and $(d_1,\dots, d_{N+M-1}+1, d_{N+M}) \in \fkD(K_N \triangle \calM)$.
	\end{enumerate}
	
	If $d_{N+M+1} = 1$, then we claim $(d_1,\dots, d_{N+M}) \in \fkD(G)$.
	Let $1 \leq i_1 < \dots < i_k \leq N+M$.
	If $i_k < N+M-1$, then
	\[
		d_{i_1} + \cdots + d_{i_k} < \left|\bigcup_{j=1}^k\calN_{D(K_N \triangle \calM')}(i_j)\right| = \left|\bigcup_{j=1}^k\calN_{D(K_N \triangle \calM)}(i_j)\right|.
	\]
	If $i_k \geq N+M-1$, then 
	\[
		\begin{aligned}
			d_{i_1} + \cdots + d_{i_k} 
				&= d_{i_1} + \cdots + d_{i_k} + 1 - 1 \\
				&< \left|\bigcup_{j=1}^k\calN_{D(K_N \triangle \calM')}(i_j)\right| - 1 \\
				&= \left|\bigcup_{j=1}^k\calN_{D(K_N \triangle \calM)}(i_j)\right|.
		\end{aligned}
	\]
	
	Next suppose that $d_{N+M+1} = 0$. 
	If we also have $d_{N+M} = d_{N+M-1} = 0$, then we would have
	\[
		N+M = d_1 + \cdots + d_{N+M+1} = d_1 + \cdots + d_{N+M-2} < \left|\bigcup_{j=1}^{N+M-2}\calN_{D(K_N \triangle \calM')}(i_j)\right| = N+M,
	\]
	which is a contradiction.
	Thus, at least one of $d_{N+M}, d_{N+M-1}$ is positive.
	Without loss of generality we assume $d_{N+M} > 0$ and will show that $(d_1,\dots, d_{N+M-1},  d_{N+M}-1)$ is in $\fkD(K_N \triangle \calM)$.
	
	For notational convenience, let
	\[
		d' = (d'_1,\dots,d'_{N+M}) = (d_1,\dots, d_{N+M-1},  d_{N+M}-1).
	\]
	Consider a sum of the form $d'_{i_1} + \cdots + d'_{i_k}$.
	If $i_k < N+M-1$, then the neighbors of $i_j$ are the same in $D(K_N \triangle \calM')$ and $D(K_N \triangle \calM)$, so the corresponding draconian inequality immediately holds. 
	If $i_k = N+M-1$, then
	\[
		\begin{aligned}
			d'_{i_1} + \cdots + d'_{i_k} - 1 &= d_{i_1} + \cdots + d_{i_{k-1}} + d_{N+M} - 1 \\
				&< \left|\bigcup_{j=1}^k \calN_{D(K_N \triangle \calM')}(i_j)\right| - |\{\overline{N+M+1}\}| \\
				&= \left|\bigcup_{j=1}^k \calN_{D(K_N \triangle \calM)}(i_j)\right|.
		\end{aligned}
	\]
	A nearly-identical argument holds if $i_k = N+M$ and $i_{k-1} = N+M-1$, so we omit it.
	Lastly, if $i_k = N+M$ and $i_{k-1} < N+M-1$, then
	\[
		\begin{aligned}
			d'_{i_1} + \cdots + d'_{i_k}
				&\leq d_{i_1} + \cdots + d_{i_k} + d_{N+M-1} - 1 \\
				&< \left|\left(\bigcup_{j=1}^k \calN_{D(K_N \triangle \calM')}(i_j)\right)\cup\calN_{D(K_N \triangle \calM')}(N+M-1)\right| - |\{\overline{N+M+1}\}| \\
				&= \left|\bigcup_{j=1}^k \calN_{D(K_N \triangle \calM')}(i_j)\right| - |\{\overline{N+M+1}\}| \\
				&= \left|\bigcup_{j=1}^k \calN_{D(K_N \triangle \calM)}(i_j)\right|.
		\end{aligned}
	\]
	Thus, $d' \in \fkD(K_N \triangle \calM)$, as desired.
	Because a symmetric argument holds if $d_{N+M-1} > 0$, this completes the case in which $d_{N+M+1} = 0$.
	
	Finally, suppose $d_{N+M+1} = 2$.
	Showing that $(d_1,\dots, d_{N+M-1}+1, d_{N+M}+1, 0) \in \scrBtri_{K_N \triangle \calM}(e)$ is equivalent to showing that $(d_1,\dots, d_{N+M-1}, d_{N+M}+1) \in \fkD(K_N \triangle \calM)$, and it is this latter statement that we will establish.
	For notational convenience, set 
	\[
		d'' = (d''_1,\dots,d''_{N+M}) = (d_1,\dots, d_{N+M-1}, d_{N+M}+1).
	\]
	
	Consider a sum of the form $d''_{i_1} + \cdots + d''_{i_k}$ where $1 \leq i_1 < \cdots < i_k \leq N+M$. 
	As before, if $i_k < N+M-1$, then the neighbors of $i_j$ are the same in $D(K_N \triangle \calM')$ and $D(K_N \triangle \calM)$, so the corresponding draconian inequality immediately holds. 
	If $i_k = N+M-1$, then, using the fact that $\calN_{D(K_N \triangle \calM')}(N+M+1) \subseteq \calN_{D(K_N \triangle \calM)}(N+M-1)$,
	\[
		\begin{aligned}
			d''_{i_1} + \cdots + d''_{i_k} &< d_{i_1} + \dots + d_{i_{k-1}} + d_{N+M-1} + 1 (+1 - 1) \\
				&< \left|\left(\bigcup_{j=1}^k \calN_{D(K_N \triangle \calM')}(i_j)\right) \cup \calN_{D(K_N \triangle \calM')}(N+M+1)\right| - |\{\overline{N+M+1}\}| \\
				&= \left|\bigcup_{j=1}^k \calN_{D(K_N \triangle \calM')}(i_j)\right| - |\{\overline{N+M+1}\}| \\
				&= \left|\bigcup_{j=1}^k \calN_{D(K_N \triangle \calM)}(i_j)\right|.
		\end{aligned}
	\]
	If $i_k = N+M$, then the same reasoning as in the case $i_k = N+M-1$ applies, with the only change being that the first strict inequality becomes a weak inequality.
	In all instances, the corresponding $D(K_N \triangle \calM)$-inequalities hold, showing that $d'' \in \fkD(K_N \triangle \calM)$.
	A symmetric argument also shows that $(d_1,\dots,d_{N+M-1}+1,d_{N+M}) \in \fkD(K_N \triangle \calM)$.
	This completes the case for $d_{N+n+1} = 2$.
	
	The completion of the above three cases shows that
	\[
		\scrAtri_{K_N \triangle \calM}(e) \uplus \scrBtri_{K_N \triangle \calM}(e) \uplus \scrCtri_{K_N \triangle \calM}(e) = \fkD(K_N \triangle \calM').
	\]
	Therefore, by our inductive assumption and the fact that $\alphatri,\betatri,\gammatri$ are all injections, 
	\[
		\begin{aligned}
			|\fkD(K_N \triangle \calM')| &= |\scrAtri_{K_N \triangle \calM}(e)| + |\scrBtri_{K_N \triangle \calM}(e)| + |\scrCtri_{K_N \triangle \calM}(e)| \\
				&= 3|\fkD(K_N \triangle \calM)| \\
				&= 3^{M+1}\binom{2(N-1)}{N-1},
		\end{aligned}
	\]
	as desired.
\end{proof}

\section{Deleting edges from $K_N$}\label{sec: deleting from complete}

In this section, we will obtain formulas for $\NVol(\APQ_G)$ where $G$ is constructed in a way different than that of Section~\ref{sec: extending triangles}.
Rather than introduce new new vertices and edges, we will instead remove edges.

\subsection{Deleting a path}

In this section, we will examine what happens when we delete a path from the complete graph $K_N$.
It will help to introduce some new notation.
If $a \leq b$ are integers, let $[a,b]$ denote the set $\{a,a+1,\dots,b\}$.

\begin{lem}\label{lem: deleting path}
	Let $N \geq 4$ and let $P$ be the length-$M$ path on $[N-M,N]$ where $uv$ is an edge if and only $|u-v|=1$.
	Set
	\[
		\calP_1 = \{ \e_i + (N-2)\e_j \in \RR^N \mid 1 \leq i \leq N,\, N-M+1 \leq j \leq N-1\}
	\]
	and
	\[
		\calP_2 = \{ r\e_i + (N-1-r)\e_{i+2} \in \RR^N \mid 0 \leq r \leq N-1,\, N-M \leq i  \leq N-2\}.
	\]
	Then $\fkD(K_N) \setminus \fkD(K_N \setminus E(P)) = \calP_1 \cup \calP_2$.
\end{lem}

\begin{proof}
	First, it is clear that $\calP_1 \cup \calP_2 \subseteq \fkD(K_N)$ since each sequence in $\calP_1 \cup \calP_2$ is a weak composition of $N-1$ into $N$ parts, and $|\calN_{D(K_N)}(i)| = N$ for each $i \in [N]$.
	For ease of notation as we continue, let $G = K_N \setminus E(P)$.
	
	For each vertex $i \in [N-M+1,N-1]$, $\left|\calN_{D(G)}(i)\right| = N-2$, since neither of $\overline{i-1}$ nor $\overline{i+1}$ are a neighbor of $i$ in $D(G)$.
	Hence, any sequence in $\fkD(K_N)$ in which the $i$th coordinate is at least $N-2$, where $N-M+1 \leq i \leq N-1$, cannot be a sequence in $\fkD(G)$; these are the sequences in $\calP_1$.
	Moreover, we further know that for each $i \in [N-M+1,N-1]$, $\left|\calN_{D(G)}(i-1)\cup\calN_{D(G)}(i+1)\right| = N-1$. 
	It follows that any sequence in $\fkD(K_N)$ in which the $(i-1)$th and $(i+1)$th coordinates sum to $N-1$, where $N-M+1 \leq i \leq N-1$, cannot be a sequence in $\fkD(G)$; these are the sequences in $\calP_2$.
	Therefore, $(\calP_1 \cup \calP_2) \cap \fkD(G) = \emptyset$, and $\calP_1 \cup \calP_2 \subseteq \fkD(K_N) \setminus \fkD(K_N \setminus E(P))$.
	
	We now need to show that if $c = (c_1,\dots,c_N) \in \fkD(K_N) \setminus \fkD(G)$, then $c \in \calP_1 \cup \calP_2$.
	There are several cases to consider.
	
	First suppose $c_\ell > 0$ for exactly one $\ell \in [N]$, and consider a sum of the form $c_{i_1} + \cdots + c_{i_k}$ with $1 \leq i_1 < \cdots < i_k \leq N$.
	If $i_j \neq \ell$ for all $j$, then the inequality
	\[
		c_{i_1} + \cdots + c_{i_k} = 0 < \left|\bigcup_{j=1}^k \calN_{D(K_N \setminus E(G))}(i_j)\right| \leq \left|\bigcup_{j=1}^k \calN_{D(K_N)}(i_j)\right|
	\]
	always holds.
	On the other hand, if $i_j = \ell$ for some $j$, then the inequalities of the form
	\[
		c_{i_1} + \cdots + c_{i_k} = c_\ell = N-1 < \left|\calN_{D(K_N \setminus E(G))}(\ell)\right| = \left|\bigcup_{j=1}^k \calN_{D(K_N \setminus E(G))}(i_j)\right|
	\]
	all hold if and only if $\ell < N-M$.
	Therefore, the sequences $(N-1)\e_{N-M},\dots,(N-1)\e_{N} \in \fkD(K_N) \setminus \fkD(G)$.
	The sequences $(N-1)\e_{N-M}$ and $(N-1)\e_N$ are in $\calP_2$, and the rest are in $\calP_1$.

	Next suppose that there are exactly two indices $1 \leq \ell < m \leq N-1$ for which $c_\ell, c_m > 0$.
	There are three ways in which a $D(G)$-draconian inequality can fail to hold:
	\begin{enumerate}
		\item if $\left|\calN_{D(G)}(\ell)\right| \leq c_\ell$;
		\item if $\left|\calN_{D(G)}(m)\right| \leq c_m$; and
		\item if $\left|\calN_{D(G)}(\ell)\cup\calN_{D(G)}(m)\right| \leq N-1$.
	\end{enumerate}
	
	We first examine when $\left|\calN_{D(G)}(\ell)\right| \leq c_\ell$.
	By construction, we know that for any $i \in [N]$, $N-2 \leq \left|\calN_{D(G)}(i)\right|$. 
	If $\left|\calN_{D(G)}(\ell)\right| = N-1$, then this would mean $c_\ell = N-1$.
	This would require $c_m = 0$, which is a contradiction to the case we are in.
	So, the only way to have $\left|\calN_{D(G)}(\ell)\right| \leq c_\ell$ would be if $c_\ell = N-2$, in which case $c_m = 1$.
	This is only possible when $\ell \in [N-M+1,N-1]$; each of the resulting sequences are in $\calP_1$.
	A symmetric argument holds if $\left|\calN_{D(G)}(m)\right| \leq c_m$.
	
	Now suppose $\left|\calN_{D(G)}(\ell)\cup\calN_{D(G)}(m)\right| \leq N-1$; we may freely assume $\ell < m$.
	It is immediate that for this to happen, we have $\ell, m \in [N-M,N]$.
	Moreover, it is straightforward to see that if $m - \ell \neq 2$, then $\left|\calN_{D(G)}(\ell)\cup\calN_{D(G)}(m)\right| = N$, hence, $m = \ell + 2$.
	These two coordinates in $c$ must sum to $N-1$; all such sequences satisfying these conditions are in $\calP_2$.
	Therefore, if $c \in \fkD(K_N) \setminus \fkD(G)$ has exactly two nonzero coordinates, then $c \in \calP_1 \cup \calP_2$.

	Lastly, suppose $c \in \fkD(K_N)$ has at least three nonzero coordinates.
	This time, if $c_\ell, c_m, c_n > 0$ with $1 \leq \ell < m < n \leq N$, then each is at most $N-3$.
	Also, the sum of any two will be at most $N-2$, and the union of the two corresponding sets of neighbors in $D(G)$ will have at least $N-1$ elements.
	Continuing via now-routine arguments, one sees that $c \in \fkD(G)$.
	Therefore, we have established the reverse inclusion $\fkD(K_N) \setminus \fkD(G) \subseteq \calP_1 \cup \calP_2$, completing our proof.
\end{proof}

Lemma~\ref{lem: deleting path} is enough to determine what happens when a path is deleted from a complete graph.

\begin{thm}
	Let $N \geq 4$ and $0 \leq M < N$.
	If $P$ is a length-$M$ path in $K_N$, then
	\[
		\NVol(\APQ_{K_N \setminus E(P)}) = \binom{2(N-1)}{N-1} - (2N - 4)(M - 1) + 4.
	\]
\end{thm}

\begin{proof}
	Since $N \geq 4$, $K_N \setminus E(P)$ is connected, and we may apply Theorem~\ref{thm: translation}.
	By Remark~\ref{rmk: relabeling invariant}, we assume that the edges of $P$ can be written as $ij$ where $j=i+1$ and $N-M \leq i \leq N-1$.
	With the help of Lemma~\ref{lem: deleting path}, we know that this is
	\[
		|\fkD(K_N)| - |\calP_1 \cup \calP_2| = \binom{2(N-1)}{N-1} - |\calP_1| - |\calP_2| + |\calP_1 \cap \calP_2|.
	\] 
	
	It is immediate to see that $|\calP_1| = N(M-1)$.
	In $\calP_2$ there is some redundancy, since, each vector of the form $0\e_i + (N-1)\e_{i+2}$ with $i \in [N-M,N-4]$ is also obtained as $(N-1)\e_{i+2} + 0\e_{i+4}$.
	Thus, $|\calP_2| = N(M-1) - (M-3)$.
	Computing $|\calP_1 \cap \calP_2 = 2(M-1)+(M-2)$ is similarly straightforward, so we omit the details.
	Therefore,
	\[
		\begin{aligned}
			|\fkD(K_N \setminus E(P))| &= \binom{2(N-1)}{N-1} - \left(N(M-1) + N(M-1) - (M-3) + \left(2(M-2)+(M-1)\right)\right) \\
				&= \binom{2(N-1)}{N-1} - (2N - 4)(M - 1) + 4,
		\end{aligned}
	\]
	as desired.
\end{proof}

\subsection{Deleting a cycle}

In this section we examine what happens when the edges of a cycle are deleted from $K_N$.
To prove our main result, we establish more notation.
Throughout this section, let $N \geq 5$ and let $C_M^{\circ} = (V_M^{\circ}, E_M^{\circ})$ be an $M$-element cycle subgraph of $K_N$. 
Label the elements of $V^\circ$ by $v_0,\dots,v_{M-1}$ so that $v_i, v_j$ form an edge if and only if $j-i = \pm 1 \pmod{M}$. 
We then denote by $\scrXcir(C_M^\circ)$ the set of sequences
\[
	\scrXcir(C_M^\circ) = \{(N-2)\e_{v_i} + \e_j \in \RR^N \mid 0 \leq i \leq M-1,\, j \in [N]\}.
\]
Denote by $\scrYcir(C_M^{\circ})$ the set of sequences
\[
	\scrYcir(C_M^{\circ}) = \{r\e_{v_i} + (N-1-r)\e_{v_{i+2}} \in \RR^N \mid v_i \in V_M^\circ,\, 2 \leq r \leq N-3\}.
\]
where the subscripts on the indices are taken modulo $M$.
When $M = 4$, we will need the following third set of sequences:
\[
	\scrZcir(C_M^\circ) = \{r\e_{v_i} + (N-2-r)\e_{v_{i+2}} + \e_s \mid v_i \in V_M^{\circ}, \, 1 \leq r \leq N-3, \, s \in [N] \setminus \{v_i,v_{i+2}\}\}.
\]

\begin{lem} \label{lem: disjoint}
	The sets $\scrXcir(C_M^\circ)$, $\scrYcir(C_M^{\circ})$ and $\scrZcir(C_M^\circ)$ are pairwise disjoint.
\end{lem}

\begin{proof}
	Note that each sequence in $\scrXcir(C_M^\circ)$ has a nonzero entry of $N-2$, while each entry of each sequence in $\scrYcir(C_M^\circ) \cup \scrYcir(C_M^\circ)$ is at most $N-3$.
	Thus, $\scrXcir(C_M^\circ) \cap \scrYcir(C_M^\circ) = \emptyset$ and $\scrXcir(C_M^\circ) \cap \scrZcir(C_M^\circ) = \emptyset$. 
	To see that $\scrYcir(C_M^\circ) \cap \scrZcir(C_M^\circ) = \emptyset$, note that each sequence in $\scrYcir(C_M^\circ)$ has exactly two nonzero entries while each sequence in $\scrZcir(C_M^\circ)$ has exactly three nonzero coordinates.
	Therefore, the three sets are pairwise disjoint.
\end{proof}

\begin{lem} \label{lem: lost sequences}
	For each $N \geq 5$ and $3 \leq M \leq N$,
	\[
		\fkD(K_N) \setminus \fkD(K_N \setminus E_M^\circ) = 
			\begin{cases}
				\scrXcir(C_M^\circ) \uplus \scrYcir(C_M^\circ) & \text { if } M \neq 4, \\
				\scrXcir(C_4^\circ) \uplus \scrYcir(C_4^\circ) \uplus \scrZcir(C_4^\circ) & \text{ if } M = 4.
			\end{cases}
	\]
\end{lem}

\begin{proof}
	We will first show that $\scrXcir(C_M^\circ) \cup \scrYcir(C_M^\circ) \subseteq \fkD(K_N) \setminus \fkD(K_N \setminus E_M^\circ)$ for all $M$, and that $\scrZcir(C_4^\circ) \subseteq \fkD(K_N) \setminus \fkD(K_N \setminus E_4^\circ)$.
	By Remark~\ref{rmk: relabeling invariant}, we may assume that $V_M^\circ = [N-M+1,N]$ and $ij \in E_M^\circ$ if and only if $j - i = \pm 1 \pmod{M}$.
	Since each sequence $c \in \scrXcir(C_M^\circ) \cup \scrYcir(C_M^\circ) \cup \scrZcir(C_M^\circ)$ is a weak composition of $N-1$ into $N$ parts, we know $c \in \fkD(K_N)$.

	For any sequence $c = (c_1,\dots,c_N) \in \scrXcir(C_M^\circ)$, there exists some $N-M+1 \leq i \leq N$ such that $c_i \geq N-2$. 
	However, since we are removing the edges in $E_M^\circ$, every vertex $i \in [N-M+1,N]$ satisfies $\deg_{K_N \setminus E_M^\circ}(i) = N-3$, meaning that $c_i \geq |\calN_{D(K_N \setminus E_M^\circ)}(i)|$ for some $i$.
	This violates one of the draconian inequalities, so $c \notin \fkD(K_N \setminus E_M^\circ)$, and it follows that $\scrXcir(C_M^\circ) \subseteq \fkD(K_N) \setminus \fkD(K_N \setminus E_M^\circ)$.

	Next, for any $c \in \scrYcir(C_M^\circ)$, there are some $i,j \in [N-M+1,N]$ such that $c_i + c_j = N-1$ and $j-i = \pm 2 \pmod{M}$.
	However, in $K_N \setminus E_M^\circ$, neither of these vertices will have $i+1 \pmod{M}$ as a neighbor.
	Therefore,
	\[
		|\calN_{D(K_N \setminus E_M^\circ)}(i) \cup \calN_{D(K_N \setminus E_M^\circ)}(j)| \leq N-1 = c_i + c_j,
	\]
	implying $c \notin \fkD(K_N \setminus E_M^\circ)$.
	It again follows that $\scrYcir(C_M^\circ) \subseteq \fkD(K_N) \setminus \fkD(K_N \setminus E_M^\circ)$.
	
	When $M=4$, deleting the edges in $C_4^\circ$ means $i$ and $j$ will have \emph{neither} $i+1$ nor $j+1$ as common neighbors when $i,j \in [N-3,N]$ and $j - i = \pm 2 \pmod{M}$.
	Using analogous reasoning as in the previous paragraph, we obtain $\scrZcir(C_4^\circ) \subseteq \fkD(K_N) \setminus \fkD(K_N \setminus E_4^\circ)$.
	It now remains to establish the reverse inclusions.

	Let $c \in \fkD(K_N) \setminus \fkD(K_N \setminus E_M^\circ)$.
	First, if $c$ only has one nonzero entry, then the value of that entry is $N-1$.
	Since $|\calN_{D(K_N \setminus E_M^\circ)}(i)| \leq N-1$ if and only if $i \in [N-M+1,N]$, then $c = (N-1)\e_i$ for some $i \in [N-M+1,N]$.
	Thus, $c \in \scrXcir(C_M^\circ)$.
	
	Now suppose $c$ has two nonzero entries, $c_i$ and $c_j$ with $c_i \leq c_j$.
	If $c_j = N-2$, then, by an argument similar to that of the previous paragraph, $c \in \scrXcir(C_M^\circ)$.
	Otherwise, $2 \leq c_i \leq c_j \leq N-3$.
	We already know that $\deg_{D(K_N \setminus E_M^\circ)}(\ell) \in \{N-2,N\}$ for each $\ell \in [N]$.
	Because $N \geq 5$, this means
	\[
		2 = c_i \leq c_j = N-3 < N-2 \leq \deg_{D(K_N \setminus E_M^\circ)}(\ell)
	\]
	for all $\ell$, namely, when $\ell \in \{i,j\}$.
	Hence, since $c \in \fkD(K_N) \setminus \fkD(K_N \setminus E_M^\circ)$, we know that the following inequality holds:
	\[
		N-1 = c_i + c_j \geq |\calN_{D(K_N \setminus E_M^\circ)}(i) \cup \calN_{D(K_N \setminus E_M^\circ)}(j)|.
	\]
	In order for the union to have fewer than $N$ elements, we need $i,j \in [N-M-1,N]$ and $j - i = \pm 2 \pmod{M}$.
	Therefore, $c \in \scrYcir(C_M^\circ)$.
	
	Now suppose that $c$ has three nonzero entries, $c_i,c_j,c_k$.
	In this case,
	\[
		c_\ell \leq N-3 < N-2 \leq |\calN_{D(K_N \setminus E_M^\circ)}(\ell)|
	\]
	for each $\ell = i,j,k$, and
	\[
		c_i + c_j + c_k = N-1 < N = \left|\bigcup_{\ell = i,j,k} \calN_{D(K_N \setminus E_M^\circ)}(\ell)\right|.
	\]
	Thus, since we know $c \in \fkD(K_N) \setminus \fkD(K_N \setminus E_M^\circ)$, an inequality of the following form must hold:
	\begin{equation}\label{eq: equality1}
		c_i + c_j \geq \left|\bigcup_{\ell = i,j} \calN_{D(K_N \setminus E_M^\circ)}(\ell)\right|
	\end{equation}
	Note, though, that we always have
	\begin{equation}\label{eq: equality2}
		c_i + c_j \leq N-2 \leq \left|\bigcup_{\ell = i,j} \calN_{D(K_N \setminus E_M^\circ)}(\ell)\right|
	\end{equation}
	so that the only way for both \eqref{eq: equality1} and \eqref{eq: equality2} to hold is if all of the inequalities are equalities.
	This occurs if and only if both $i,j \in [N-M+1,N]$ and $i$ and $j$ have the same neighbors in $K_N \setminus E_M^\circ$. 
	These conditions hold if and only if $M=4$ and and $j-i = \pm 2 \pmod{M}$, and subsequently $c_k = 1$; hence, $c \in \scrZcir(C_4^\circ)$.
	This establishes the desired reverse inclusions.
	By Lemma~\ref{lem: disjoint}, we are done.
\end{proof}

We are now ready to prove our main result of the section.

\begin{thm}
	For all $N \geq 5$ and $0 \leq M \leq N$,
	\[
		\NVol(\APQ_{K_N \setminus E_M^\circ}) =
			\begin{cases} 
				\binom{2(N-1)}{N-1} - 2M(N-2) &  \text{ if }M \neq 4 \\
				\binom{2(N-1)}{N-1} - 2(N+1)(N-2) & \text{ if } M = 4.
			\end{cases}
	\]
\end{thm}

\begin{proof}
	Since $N \geq 5$, $K_N \setminus E_M^\circ$ is connected, and we may apply Theorem~\ref{thm: translation}.
	By Lemma~\ref{lem: lost sequences}, the desired formula comes from computing
	\[
		|\fkD(K_N \setminus E_M^\circ)| =
			\begin{cases} 
				|\fkD(K_N)| - |\scrXcir(C_M^\circ)| - |\scrYcir(C_M^\circ)| & \text{ if } M \neq 4, \\
				|\fkD(K_N)| - |\scrXcir(C_4^\circ)| - |\scrYcir(C_4^\circ)| - |\scrZcir(C_4^\circ)| & \text{ if } M = 4.
			\end{cases}
	\]
	We already know that $|\fkD(K_N)| = \binom{2(N-1)}{N-1}$.
	Elementary counting arguments and algebra will verify that 
	\[
		|\scrXcir(C_M^\circ)| = MN  \quad \text{and} \quad |\scrYcir(C_M^\circ)| = M(N-4)
	\]
	for all $M$.
	Thus, when $M \neq 4$, we get
	\[
		|\fkD(K_N)| - |\scrXcir(C_M^\circ)| - |\scrYcir(C_M^\circ)| = \binom{2(N-1)}{N-1} - MN - M(N-4) = \binom{2(N-1)}{N-1} - 2M(N-2),
	\]
	as claimed.
	When $M=4$, it is also not difficult to directly show that $|\scrZcir(C_4^\circ)| = 2(N-3)(N-2)$, so that
	\[
		\begin{aligned}
			\fkD(K_N \setminus E_4^\circ)
				&= |\fkD(K_N)| - |\scrXcir(C_4^\circ)| - |\scrYcir(C_4^\circ)| - |\scrZcir(C_4^\circ)| \\
				&= \binom{2(N-1)}{N-1} - (4N + 4(N-4) + 2(N-3)(N-2)) \\
				&= \binom{2(N-1)}{N-1} - 2(N+1)(N-2),
		\end{aligned}
	\]
	as claimed.	
\end{proof}


\bibliographystyle{plain}
\bibliography{references}

\end{document}